\documentclass[12pt, a4paper]{article}

\usepackage{amsmath,amsfonts,amssymb}
\usepackage{hyperref}
\usepackage[capitalise]{cleveref}

\usepackage[normalem]{ulem}
\usepackage{xcolor}

\usepackage{geometry}
\usepackage{bbm}
\usepackage{mathtools} 

\newcommand{\E}{\mathbb{E}}
\newcommand{\F}{\mathcal{F}}

\newcommand{\PP}{\mathbb{P}}
\newcommand{\N}{\mathbb{N}}
\newcommand{\one}{\mathbbm{1}}
\newcommand{\dd}{\mathrm{d}}
\newcommand{\define}{\vcentcolon =}

\usepackage{amsthm} 

\newtheorem{theorem}{Theorem}[section]
\newtheorem{corollary}[theorem]{Corollary}

\newtheorem{proposition}[theorem]{Proposition}

\theoremstyle{definition}

\newtheorem{remark}[theorem]{Remark}

\title{Sharpness of Lenglart's domination inequality and a sharp monotone version}
\author{%
Sarah~Geiss\footnote{Technische Universit\"at Berlin, Germany. E-mail: \href{mailto:geiss@math.tu-berlin.de}{geiss@math.tu-berlin.de}}
\thanks{This author was supported by the 
Elsa-Neumann-Stipendium des Landes Berlin.}
\and 
  Michael~Scheutzow \footnote{Technische Universit\"at Berlin, Germany. E-mail: 
  \href{mailto:ms@math.tu-berlin.de}{ms@math.tu-berlin.de}}}

\date{}   

\begin{document}

\parindent0em

\maketitle

\begin{abstract} We prove that the best so far known constant $c_p=\frac{p^{-p}}{1-p},\, p\in(0,1)$ of a domination inequality, which originates to Lenglart, is sharp. In particular, we solve an open question posed by Revuz and Yor \cite{RevuzYor}.  Motivated by the application to  maximal inequalities, like e.g. the Burkholder-Davis-Gundy inequality, we also study the domination inequality under an additional monotonicity assumption. In this special case, a constant which stays bounded for $p$ near $1$ was proven by Pratelli and Lenglart. We provide the sharp constant for this case.
\end{abstract}

[0.25em]
\textbf{MSC2020 subject classifications:} 60G44, 60G40, 60G42, 60J65

\bigskip

\section{Introduction}

In this note, we prove that the best so far known constant  $c_p$ of a domination inequality, which originates to Lenglart \cite[Corollaire II]{Lenglart} (see \cref{th:lenglart}), is sharp. In particular, we  solve an open question posed by Revuz and Yor \cite[Question IV.1, p.178]{RevuzYor}.  Furthermore, motivated by the method of applying Lenglart's inequality to extend maximal inequalities to small exponents, we study Lenglart's domination inequality under an additional monotonicity assumption: A result by Pratelli \cite{Pratelli} and Lenglart \cite{Lenglart} implies (under the additional monotonicity assumption) a constant, which is bounded by $2$, and hence considerably improves the constant of Lenglart's inequality for $p$ near 1. We provide a sharp constant. The sharpness of our monotone version of Lenglart's inequality is related to a result by Wang \cite{Wang}.  \\

\hspace{1em} Let $(\Omega, \F, \PP, (\F_t)_{t\geq 0 })$ be a filtered probability space satisfying the usual conditions. The following lemma is \cite[Lemma 2.2 (ii)]{MehriScheutzow}:

\begin{theorem}[Lenglart's inequality] \label{th:lenglart}
Let $X$  and $G$ be non-negative adapted right-continuous processes, and let $G$ be in addition non-decreasing and predictable such that 
$\E[X_\tau \mid\F_0] \leq \E[G_\tau\mid\F_0] \leq \infty$ for any bounded stopping time $\tau$. Then for all $p\in(0,1)$,
\begin{equation*}
\E\bigg[\bigg(\sup_{t\geq 0} X_t\bigg)^p \,\bigg| \, \F_0 \bigg]
\leq c_p \E\bigg[\bigg(\sup_{t\geq 0} G_t\bigg)^p \, \bigg|\, \F_0 \bigg]
\end{equation*}
where $c_p \define \frac{p^{-p}}{1-p}$.
\end{theorem} 

\hspace{1em} In the original work by Lenglart \cite[Corollaire II]{Lenglart}, the inequality is proven for $c_p = \frac{2-p}{1-p}$, $p\in(0,1)$. The constant $c_p$ is improved to $\frac{p^{-p}}{1-p}$  by Revuz and Yor in \cite[Exercise IV.4.30]{RevuzYor} for continuous processes $X$ and $G$.  This result is generalized to c\`adl\`ag processes by Ren and Shen in \cite[Theorem 1]{RenShen} and is extended to a more general setting than \cite[Corollaire II]{Lenglart} by Mehri and Scheutzow \cite[Lemma 2.2 (ii)]{MehriScheutzow}. Furthermore, the growth rate of the optimal constant $c_p^{(opt)}$ for c\`adl\`ag processes has been studied (see \cite[Theorem 2]{RenShen}): It holds that $(c_p^{(opt)})^{1/p}= O(1/p)$ for $p\to 0^+$. We prove (see \cref{th:Haupt}) that $\frac{p^{-p}}{1-p}$ is sharp. \\

\hspace{1em} Lenglart's inequality yields a very short proof of the Burkholder-Davis-Gundy inequality for continuous local martingales for small exponents (see e.g. \cite[Theorem IV.4.1]{RevuzYor}): Let $(M_t)_{t\geq 0}$ be a continuous local martingale with $M_0=0$. To prove $\E[\langle M, M \rangle_t^{q/2} ] \lesssim \E[\sup_{t\geq 0} |M_t|^q]$ for $q\in(0,2)$, take 
\begin{equation*}
X_t\define\langle M, M \rangle_t, \qquad G_t\define \sup_{0\leq s \leq t} |M_s|^2.
\end{equation*}

Using that $M_t^2 - \langle M,M \rangle_t$ is a continuous local martingale, we have $\E[X_\tau] \leq \E[G_\tau]$ for any bounded stopping time $\tau$. Applying Lenglart's inequality with $p=q/2$, we obtain 
\begin{equation*}
\E[\langle M, M \rangle_t^{q/2} ] \leq c_{q/2} \E[\sup_{t\geq 0} |M_t|^q].
\end{equation*}
For $q=1$, this implies $c_{BDG,1} = c_{q/2} = 2\sqrt{2} \approx 2,8284$. The optimal BDG constant can be computed numerically for this case (see Schachermayer and Stebegg \cite{Schachermayer}) and is $c^{(opt)}_{BDG,1} \approx 1,2727$.  A better constant than $c_{q/2}$ can be achieved if we apply the following proposition due to Lenglart \cite[Proposition I]{Lenglart} and Pratelli \cite[Proposition 1.2]{Pratelli} instead:

\begin{proposition}[Lenglart, Pratelli] \label{prop:Lenglart}
Let $F$ be a concave non-decreasing function with $F(0) =0$ and let $c>0$ be a constant. Let $Y$ and $G$ be adapted non-negative right-continuous processes starting in $0$. Furthermore, let $G$ be non-decreasing and predictable. Assume that $\E[Y_\tau] \leq c \E[G_\tau]$ holds for all finite stopping times $\tau$. Then, for all finite stopping times $\tau$, we have
\begin{equation*}
\E[F(Y_\tau)] \leq (1+c)\E[F(G_\tau)].
\end{equation*}
\end{proposition}

\bigskip
 Let $X$ and $G$ be as in \cref{th:lenglart}. Assume in addition that both processes start in $0$. Then  Proposition \ref{prop:Lenglart} implies, choosing $F(x) = x^{p}$ for some $p\in(0,1)$ and optimizing over $c$, that
\begin{equation}\label{eq:PropPratelliUgl}
\E[X_\tau^p] \leq (1-p)^{-(1-p)} p^{-p}\E[G_\tau^p].
\end{equation}

Hence, Proposition \ref{prop:Lenglart} gives $c_{BDG,1} = 2$. We show that the constant of inequality \eqref{eq:PropPratelliUgl} can be improved  to $p^{-p}$ (see  \cref{th:Monotonelenglart} and Remark \ref{rmk:rmk4}), which is sharp. In particular, by the argument described above we now achieve $c_{BDG,1} = \sqrt{2} \approx 1,4142$. 
For the right-hand side of the BDG inequality $\E[\sup_{t\geq 0}
|M_t|] \lesssim \E[\langle M, M \rangle_t^{1/2}]$,  the monotone version of
Lenglart's inequality does not yield a sharper constant than the normal Lenglart's
inequality. \\

\hspace{1em} Lenglart's inequality is frequently applied to extrapolate   maximal inequalities to smaller exponents (see e.g. \cite{6}, \cite{MarinelliRoeckner}, \cite{NeervenZhu}, \cite{7} and \cite{2}).
Furthermore, Lenglart's inequality is a useful tool for proving stochastic Gronwall inequalities (see e.g. \cite{4} and \cite{MehriScheutzow}) and more generally studying SDEs (see e.g. \cite{3} and \cite{5}). In many of the application examples listed above, the additional assumption, that $X$ is non-decreasing is satisfied. Hence, instead, \cref{th:Monotonelenglart} could be applied, 
improving the constant considerably for $p$ near $1$.

\section{Main results}

We assume, unless otherwise stated, that all processes are defined on an underlying filtered probability space $(\Omega, \F, \PP, (\F_t)_{t\geq 0})$ which satisfies the usual conditions. \\

The following theorem answers the open question posed by Revuz and Yor \cite[Question IV.1, p.178]{RevuzYor}.

\begin{theorem}[Sharpness of Lenglart's inequality]\label{th:Haupt}
For all $p\in(0,1)$, there exist families of continuous processes $X^{(n)}=(X^{(n)}_t)_{t\geq 0}$ and $G^{(n)}=(G^{(n)}_t)_{t\geq 0}$ (depending on $p$) which satisfy the assumptions of \cref{th:lenglart} such that

\begin{equation} \label{eq:lemmahaupt}
\frac{p^{-p}}{1-p} = \lim_{n\to\infty} \frac{ \E\bigg[\big(\sup_{t\geq 0 }X^{(n)}_{t}\big)^p\bigg]}{ \E\bigg[\big(\sup_{t\geq 0} G^{(n)}_t\big)^p\bigg]}.
\end{equation}

In particular, the constant $c_p = \frac{p^{-p}}{1-p}$ in \cref{th:lenglart} is sharp.
\end{theorem}

As explained in the introduction, the application to maximal inequalities motivates us to consider the following monotone version of Lenglart's inequality. We assume in addition that $X$ is non-decreasing and obtain a considerably improved constant for $p$ near $1$. 

\begin{theorem}[Sharp monotone Lenglart's inequality] \label{th:Monotonelenglart}
Let $X$  and $G$ be non-decreasing non-negative adapted right-continuous processes, and let $G$ be in addition predictable such that 
$\E[X_\tau\mid\F_0] \leq \E[G_\tau \mid\F_0] \leq \infty$ for any bounded stopping time $\tau$. Then for all $p\in(0,1)$,
\begin{equation}\label{eq:Monotonelenglart}
\E\bigg[\bigg(\sup_{t\geq 0} X_t\bigg)^p \,\bigg| \, \F_0 \bigg]
\leq p^{-p} \, \E\bigg[\bigg(\sup_{t\geq 0} G_t\bigg)^p \, \bigg|\, \F_0 \bigg].
\end{equation}

Furthermore, for all $p\in(0,1)$ there exist continuous processes $\tilde{X}=(\tilde{X}_t)_{t\geq 0}$ and $ \tilde{G}=(\tilde{G}_t)_{t\geq 0}$, satisfying the assumptions above such that
\begin{equation*}
p^{-p} = \lim_{n\to\infty} \frac{ \E\bigg[\big(\sup_{t\geq 0 } \tilde{X}_{t\wedge n}\big)^p\bigg]}{ \E\bigg[\big(\sup_{t\geq 0 } \tilde{G}_{t\wedge n}\big)^p\bigg]}.
\end{equation*}

In particular, the constant $p^{-p}$ is sharp.
\end{theorem} 

\begin{remark}
Inequality \eqref{eq:Monotonelenglart} is a sharpened special case of Proposition \ref{prop:Lenglart}, its proof is a modification of the proof of \cite[Proposition 1.2]{Pratelli}. The theorem generalizes a result by Garsia \cite[Theorem III.4.4, page 113]{Garsia}. In \cite[Theorem 2]{Wang}, Wang proved that \cite[Theorem III.4.4, page 113]{Garsia} is sharp. Hence, by translating his result from discrete to continuous time proves sharpness of $p^{-p}$.
\end{remark}

\begin{remark}\label{rmk:rmk4}
\cref{th:Monotonelenglart} can be also applied when $X$ is not non-decreasing. In that case, the theorem implies for any stopping time $\tau$ the inequality $\E[X_\tau^p] \leq p^{-p} 
\, \E[G_\tau^p]$. This can by seen by defining $\hat{X}_t := X_\tau \one_{[\tau,\infty)}(t)$ for all $t\geq 0 $ and noting that $(\hat{X}_t)_{t\geq 0}$  and $(G_{t\wedge \tau})_{t \geq 0}$  satisfy the assumptions of \cref{th:Monotonelenglart}.  
\end{remark}

\begin{remark}
In \cref{th:Monotonelenglart}, the assumption that $G$ is right-continuous and predictable can be replaced by the assumption that $G$ is left-continuous and adapted.
\end{remark}

\begin{remark}
A key part of the proof of Lenglart's inequality is the inequality
\begin{equation*}
\PP\bigg(\sup\limits_{t\geq 0} X_t > c\,\bigg|\,  F_0\bigg)\leq \frac{1}{c} \E\bigg[\sup\limits_{t\geq 0 } G_t \wedge d \,\bigg|\, \F_0\bigg] + \PP\bigg( \sup\limits_{t\geq 0} G_t \geq d \,\bigg|\, \F_0 \bigg)
\end{equation*}
for all $c,d>0$. If $X$ is non-decreasing, this can be improved to
\begin{equation*}
\frac{1}{c} \E\bigg[\sup\limits_{t\geq 0 } X_t \wedge c \,\bigg|\, \F_0\bigg] \leq \frac{1}{c} \E\bigg[\sup\limits_{t\geq 0 } G_t \wedge d \,\bigg|\, \F_0\bigg] + \PP\bigg( \sup\limits_{t\geq 0} G_t \geq d \,\bigg|\, \F_0 \bigg),
\end{equation*}
which is used to prove the monotone version of Lenglart's inequality.
\end{remark}

\begin{remark} If $G$ is not predictable and no further assumptions are made, then there exists no finite constant in inequality \eqref{eq:Monotonelenglart}. An example which demonstrates this can be found in \cite[Remarque after Corollaire II]{Lenglart}. 
\end{remark}

\cref{th:lenglart}, \cref{th:Haupt},
and \cref{th:Monotonelenglart} also hold in discrete time. Here, sharpness of $p^{-p}$ follows immediately from \cite[Theorem 2]{Wang}.

\begin{corollary}[Discrete Lenglart's inequality] \label{cor:discrete}
Let $(X_n)_{n\in\N_0}$  and $(G_n)_{n\in\N_0}$ be non-negative adapted processes, and let $G$ be in addition non-decreasing and predictable such that 
$\E[X_\tau \mid\F_0] \leq \E[G_\tau \mid\F_0] \leq \infty$ for any bounded stopping time $\tau$. Then for all $p\in(0,1)$,
\begin{equation}\label{eq:A}
\E\bigg[\bigg(\sup_{n\in\N_0} X_n \bigg)^p \,\bigg| \, \F_0 \bigg]
\leq c_p \, \E\bigg[\bigg(\sup_{n\in\N_0} G_n \bigg)^p \, \bigg|\, \F_0 \bigg],
\end{equation}
where $c_p \define \frac{p^{-p}}{1-p}$ and the constant $c_p$ is sharp. \\

If we assume in addition, that $(X_n)_{n\in\N_0}$ is non-decreasing, then we have 
\begin{equation}\label{eq:B}
\E\bigg[\bigg(\sup_{n\in\N_0} X_n \bigg)^p \,\bigg| \, \F_0 \bigg]
\leq p^{-p} \, \E\bigg[\bigg(\sup_{n\in\N_0} G_n \bigg)^p \, \bigg|\, \F_0 \bigg]
\end{equation}
and the constant $p^{-p}$ is sharp.
\end{corollary}

\section{Proof of \cref{th:Haupt}}
\begin{proof}[Proof of \cref{th:Haupt}] Choose an arbitrary $p\in(0,1)$ for the remainder of this proof. First, we define non-decreasing processes $\tilde{X}=(\tilde{X}_t)_{t\geq 0}$ and  $\tilde{G}=(\tilde{G}_t)_{t\geq 0}$ which satisfy the assumptions of \cref{th:lenglart}, such that
\begin{equation*}
p^{-p} = \lim_{n\to\infty}\frac{\E\big[\big(\sup_{t\geq 0} \tilde{X}_{t\wedge n}\big)^p\,\big]}{\E\big[\big(\sup_{t\geq 0} \tilde{G}_{t\wedge n}\big)^p \,
\big]}.
\end{equation*}
To obtain the extra factor $(1-p)^{-1}$, we modify $\tilde{X}$ and $\tilde{G}$ using an independent Brownian motion: This gives us the families $\{(X^{(n)}_t)_{t\geq 0},n\in\N\}$ and $\{(G^{(n)}_t)_{t\geq 0},n\in\N \}$. \\

Note that if we have non-negative random variables $X_{RV}:=1$ and $G_{RV}$ with $\E[X_{RV}] = \E[G_{RV}]$, then we obtain
 $\E[X_{RV}^p] >> \E[G_{RV}^p]$  for example by choosing $G_{RV}$ to be very large on a set with small probability and everywhere else $0$. Keeping this in mind, we construct $\tilde{X}$ and  $\tilde{G}$ as follows: Let $Z$ be an exponentially distributed random variable on a complete probability space $(\Omega, \F, \PP)$ with $\E[Z]=1$. Set 
\begin{equation*}
A:[0,\infty)\to[0,\infty), \quad t\mapsto \exp(t/p).
\end{equation*}

Define for all $t\geq 0 $
\begin{equation*}
\tilde{X}_t \define A(Z) \one_{[Z,\infty)}(t), \qquad \tilde{G}_t \define  \int_0^{t\wedge Z} A(s) \dd s.
\end{equation*}
 
Choose $\tilde{\F}_t := \sigma(\{Z \leq r\} \mid  0 \leq r \leq t)$ for all $t\geq 0$. Observe that $\tilde{X}$ and $\tilde{G}$ are non-decreasing non-negative adapted right-continuous processes, and $\tilde{G}$ is in addition continuous, hence predictable.  Furthermore, due to $Z$ being exponentially distributed, $\tilde{G}$ is the compensator of $\tilde{X}$, implying $\E[\tilde{X}_\tau] = \E[\tilde{G}_\tau]$ for all bounded $\tau$.

\bigskip

Now we use the processes  $\tilde{X}$ and $\tilde{G}$ to construct the families $\{(X^{(n)}_t)_{t\geq 0},n\in\N\}$ and $\{(G^{(n)}_t)_{t\geq 0},n\in\N \}$: Assume w.l.o.g. that there exists a Brownian motion $B$ on $(\Omega, \F, \PP)$. Let
$(\F_t)_{t\geq 0}$ be the smallest filtration satisfying the usual conditions which contains $(\tilde{\F}_t)_{t\geq 0}$ and w.r.t. which $B$ is a Brownian motion. Denote by $g_{n,n+1}:[0,\infty) \to [0,1]$ a continuous non-decreasing function such that 
\begin{equation}
g_{n,n+1}(t) = 0 \quad \forall t\leq n,\quad \text{and} \quad  g_{n,n+1}(t) = 1 \quad \forall t \geq n+1.
\end{equation}
Define:
\begin{equation*}
\begin{aligned}
\tau^{(n)} &\define \inf\{t\geq n+1 \mid \tilde{X}_{n} + (B_t - B_{n+1})\one_{\{t\geq n+1\}}  = 0 \}, \\
X^{(n)}_t & \define g_{n,n+1}(t) \tilde{X}_{n} + (B_{t\wedge \tau^{(n)}} - B_{t\wedge (n+1)}) \\
G^{(n)}_t & \define \tilde{G}_{t\wedge n}
\end{aligned}
\end{equation*}

The stopping time $\tau^{(n)}$ ensures that $X^{(n)}_t$ is non-negative. By construction, we have for every bounded $(\F_t)_{t\geq 0}$ stopping time $\tau$
\begin{equation*}
\E[X^{(n)}_\tau]  \leq \E[\tilde{X}_{\tau\wedge n} +  B_{\tau\wedge \tau^{(n)}} - B_{\tau\wedge (n+1)}] = \E[\tilde{G}_{\tau \wedge n}] = \E[G^{(n)}_\tau].
\end{equation*}

Hence, $(X^{(n)}_t)_{t\geq 0}$ and  $(G^{(n)}_t)_{t\geq 0}$  are continuous processes that satisfy the assumptions of \cref{th:lenglart}.   \\

It remains to calculate 
$\E\big[\big(\sup_{t\geq 0 }X^{(n)}_{t}\big)^p\big]$ and $\E\big[\big(\sup_{t\geq 0 }G^{(n)}_{t}\big)^p\big]$,
to show that equation \eqref{eq:lemmahaupt} is satisfied.  We have
\begin{equation}\label{eq:XpGp}
\begin{aligned}
\E[\tilde{X}^p_t] & = \int_0^\infty A(x)^p \one_{\{t\geq x\}}\exp(-x) \dd x = t, \\
\E[\tilde{G}^p_t] &  = \int_0^\infty   \bigg(\int_0^{t\wedge x} A(s) \dd s\bigg)^p \exp(-x) \dd x  \leq  p^{p}(t+1),
\end{aligned}
\end{equation}
which implies in particular that 
$\E\big[\big(\sup_{t\geq 0} G^{(n)}_t \big)^p\big] \leq p^p(n+1)$. \\

We calculate $\E\big[\big(\sup_{t\geq 0 }X^{(n)}_{t}\big)^p\big]$ using the independence of $Z$ and $B$. To this end, let $\tilde{B}$ be some Brownian motion and  consider for all $0\leq x < a^{1/p}$  the stopping times
\begin{equation*}
\sigma_{x}  : =\inf\{t \geq 0 \mid  \tilde{B}_t + x = 0\}, \quad 
\sigma_{x,a} : =\inf\{t \geq 0 \mid \tilde{B}_t + x = a^{1/p}\}.
\end{equation*}
Define the family of random variables $Y_x := \sup_{t\geq 0}  \tilde{B}_{t\wedge \sigma_x}  + x$, $x\geq 0$. Then $\E[\tilde{B}_{\sigma_x\wedge \sigma_{x,a}}]=0$ implies $\PP[ Y_x \geq a^{1/p}] = \PP[\sigma_{x,a} < \sigma_{x}] = xa^{-1/p}$, and hence
\begin{equation}\label{eq:Yxp}
\E[Y_x^p]  = x^p + \int_{x^p}^\infty \PP[ Y_x \geq a^{1/p}] \dd a = x^p + x^p\frac{p}{1-p} =\frac{x^p}{1-p}.
\end{equation}

Hence, we have by \eqref{eq:XpGp}, \eqref{eq:Yxp} and independence of $(B_t-B_{n+1})_{t\geq n+1}$ and $\F_{n+1}$:
\begin{equation*}
\begin{aligned}
\E\big[\big(\sup_{t\geq 0}X^{(n)}_t \big)^p\big] &  = 
\E\big[ \E\big[\big(\sup_{t\geq 0}X^{(n)}_t \big)^p\mid \F_{n+1}\big]\big] \\
& = \E\bigg[\frac{1}{1-p}\big(\tilde{X}_n\big)^p\bigg]\\
& = \frac{n}{1-p}.
\end{aligned}
\end{equation*}

Therefore, we have:
\begin{equation*}
c_p \geq \frac{\E[(\sup_{t\geq 0} X^{(n)}_t)^p]}{\E[(\sup_{t\geq 0} G^{(n)}_t)^p]} \geq  \frac{n}{1-p}\frac{p^{-p}}{n+1},
\end{equation*}
which implies \eqref{eq:lemmahaupt}.
\end{proof}

\section{Proof of \cref{th:Monotonelenglart}}

\begin{remark}
The following proof of inequality \eqref{eq:Monotonelenglart} is a modification of the proof of \cite[Proposition 1.2]{Pratelli}.
Sharpness of the constant can be proven using \cite[Theorem 2]{Wang}.
\end{remark}

\begin{proof}[Proof of \cref{th:Monotonelenglart}]
We first show that $p^{-p}$ is the optimal constant. Sharpness of $p^{-p}$ can be proven
by translating \cite[Theorem 2]{Wang} into continuous time. Alternatively, one can use the processes $\tilde{X}$ and $\tilde{G}$ and  the filtration $(\F_t)_{t\geq 0}$ from the proof of \cref{th:Haupt}: Equation \eqref{eq:XpGp} implies, that
\begin{equation*}
p^{-p} = \lim_{n\to\infty}\frac{\E\bigg[\bigg(\sup_{t\geq 0} \tilde{X}_{t\wedge n}\bigg)^p\,\bigg]}{\E\bigg[\bigg(\sup_{t\geq 0} \tilde{G}_{t\wedge n}\bigg)^p \,\bigg]},
\end{equation*}
and therefore that $p^{-p}$ is sharp. \\

Now we prove that inequality \eqref{eq:Monotonelenglart} holds true. We may assume w.l.o.g. that $(G_t)_{t\geq 0}$ is bounded (because it is predictable). This implies $\E[\sup_{t\geq 0} X_t] <\infty$. To shorten notation, we define
\begin{equation}
X_\infty \define \sup_{t\geq 0} X_t, \qquad G_\infty \define \sup_{t\geq 0} G_t.
\end{equation}
We use the following formulas for positive random variables $Z$
 (equation \eqref{eq:formula2} is a direct consequence of \eqref{eq:formula1}, alternatively see  also \cite[Theorem 20.1, p. 38-39]{Burkholder}):
\begin{align}
\E[Z^p \mid \F_0 ] & = \int_0^\infty \PP[Z \geq u^{1/p}\mid \F_0 ] \, \dd u,  \label{eq:formula1}\\
\E[Z^p \mid \F_0 ] & = p(1-p) \int_0^\infty \E[Z\wedge u \mid \F_0 ] \, u^{p-2} \dd u. \label{eq:formula2}
\end{align}

We will apply \eqref{eq:formula2} to $X_\infty$.  To estimate $\E[X_\infty \wedge t \mid F_0 ]$, we fix some $t,\lambda >0$ and define:
\begin{equation*}
\tau := \inf\{s \geq 0 \mid G_s \geq \lambda t\}.
\end{equation*}

Because $(G_t)_{t\geq 0}$ is predictable, there exists a sequence of stopping times $(\tau^{(n)})_{n\in\N}$ that announces $\tau$.
Therefore, we have on the set $\{G_0 < \lambda t\}$ :
\begin{equation}\label{eq:A1}
\begin{aligned}
\E[X_{\tau-}  \mid \F_0 ] & = \lim_{n\to\infty} \E[X_{\tau^{(n)}}  \mid \F_0 ] \leq \lim_{n\to\infty} \E[G_{\tau^{(n)}}  \mid \F_0 ] \\[0.5em]
& \leq \E[G_\infty \wedge \lambda t  \mid \F_0] = 
 \lambda \E[(G_\infty \lambda^{-1}) \wedge t  \mid \F_0 ].
\end{aligned}
\end{equation}

On  $\{\tau =\infty\}$ we have $\lim_{n\to\infty}X_{\tau^{(n)}}\wedge t =
X_\infty\wedge t$, which implies on the set $\{G_0 < \lambda t\}$ :
\begin{equation}\label{eq:B1}
\E[X_\infty \wedge t - X_{\tau_-} \wedge t  \mid \F_0 ] \leq t  \E[ \one_{\{\tau <
+\infty\}}\mid \F_0 ].
\end{equation}

Combining inequalities \eqref{eq:A1} and \eqref{eq:B1} gives:
\begin{equation}\label{eq:estimate}
\begin{aligned}
\E[X_\infty \wedge t \mid \F_0] 
&\leq  t \one_{\{G_0 \geq \lambda t\}}  + \big(\E[X_{\tau_-} \mid \F_0] + 
\E[X_\infty\wedge t -  X_{\tau_-}\wedge t \mid \F_0]\big)\one_{\{G_0 < \lambda
t\}} \\[0.5em]
  &\leq \lambda \E[(G_\infty \lambda^{-1}) \wedge t \mid \F_0] + t \PP[G_\infty \geq
\lambda t \mid \F_0].
\end{aligned}
\end{equation}

Applying  \eqref{eq:formula2} to  $X_\infty$ and inserting \eqref{eq:estimate} gives:

\begin{equation*}
\begin{aligned}
\E[X_\infty^p  \mid \F_0] &\leq \lambda p(1-p) \int_0^\infty \E[(G_\infty\lambda^{-1}) \wedge u  \mid \F_0] u^{p-2} \dd u   \\
& \quad + p(1-p) \int_0^\infty \PP[G_\infty \geq \lambda u  \mid \F_0] u^{p-1} \dd u.
\end{aligned}
\end{equation*}

Applying  \eqref{eq:formula1} and \eqref{eq:formula2} to $G_\infty$ in the previous inequality implies:
\begin{equation*}
\begin{aligned}
\E[X_\infty^p \mid \F_0] & \leq \lambda^{1-p} \E[G_\infty^p  \mid \F_0] + (1-p) \int_0^\infty \PP[G_\infty \geq \lambda y^{1/p}  \mid \F_0]  \dd y \\
& \leq \lambda^{-p}\big(\lambda + 1-p\big) \E[G_\infty^p  \mid \F_0 ].
\end{aligned}
\end{equation*}

Choosing $\lambda = p$ implies the assertion of the theorem.
\end{proof}

\section{Proof of Corollary \ref{cor:discrete}}
\begin{proof}[Proof of Corollary \ref{cor:discrete}]
We first prove inequalities  \eqref{eq:A} and \eqref{eq:B}:
We turn the processes $(X_n)_{n\in\N_0}$ and $(G_n)_{n\in\N_0}$ into c\`adl\`ag processes in continuous time
 as follows: Set for all $n\in\N_0, t\in[n,n+1)$:
\begin{equation*}
X_t \define X_n, \qquad G_t \define  G_n, \qquad \F_t := \F_n.
\end{equation*}
As we can approximate $(G_t)_{t\geq 0}$ by left-continuous adapted processes, it is predictable. Now \cref{th:lenglart} and \cref{th:Monotonelenglart}
immediately imply inequalities  \eqref{eq:A} and \eqref{eq:B}. \\

Sharpness of $p^{-p}$ follows from \cite[Theorem 2]{Wang}. We show that $\frac{p^{-p}}{1-p}$ is sharp. 

Let $X^{(n)}$,  $G^{(n)}$, $A$ and $(\F_t)_{t\geq 0}$ be as in proof of \cref{th:Haupt}.  Fix some arbitrary $N\in\N$. Set for all $k, n\in\N$
\begin{equation*}
\begin{aligned}
X^{(n, N)}_0 & \define X^{(n)}_0 
\qquad & X^{(n, N)}_k &\define  X^{(n)}_{k2^{-N}}, \\
G^{(n,N)}_0 &\define G^{(n)}_0 \qquad &
G^{(n,N)}_k &\define G^{(n)}_{(k-1)2^{-N}} + \int_ {(k-1)2^{-N}\wedge n}^{k 2^{-N}\wedge n}A(s)\dd s, \\
\F^{(n,N)}_0 &\define \F_0 &
\F^{(n,N)}_k &\define \F_{k2^{-N}}.
\end{aligned}
\end{equation*}
The processes $(X^{(n,N)}_k)_{k\in \N_0}$ and $(G^{(n,N)}_k)_{k\in \N_0}$ are non-negative and adapted,  $(G^{(n,N)}_k)_{k\in \N_0}$ is in addition non-decreasing and predictable. Since $G^{(n)}_{k2^{-N}}  \leq  G^{(n,N)}_k$, the processes satisfy the Lenglart domination assumption.

Hence, noting that
\begin{equation*}
\begin{aligned}
\lim_{N\to\infty}\E\bigg[ \bigg(\sup_{k\in\N_0} X^{(n,N)}_k\bigg)^p \, \bigg] &= \E\bigg[\bigg(\sup_{t\geq 0 }X^{(n)}_{t}\bigg)^p\, \bigg],\\
\lim_{N\to\infty}\E\bigg[ \bigg(\sup_{k\in\N_0} G^{(n,N)}_k\bigg)^p\, \bigg] & = \E\bigg[\bigg(\sup_{t\geq 0 } G^{(n)}_{t}\bigg)^p\, \bigg],
\end{aligned}
\end{equation*}

implies the assertion of the corollary.

\end{proof}

\bibliographystyle{amsplain} 
\bibliography{lenglartnonr.bib}

\end{document}